\documentclass[12pt]{article}
\usepackage{amsmath,amssymb,amsthm, blindtext, mathtools}

\long\def\symbolfootnote[#1]#2{\begingroup%
\def\thefootnote{\fnsymbol{footnote}}\footnote[#1]{#2}\endgroup}

\newtheorem{thm}{Theorem}[section]
\newtheorem{prop}[thm]{Proposition}

\newtheorem{cor}[thm]{Corollary}

\theoremstyle{definition}

\newtheorem{example}[thm]{Example}

\theoremstyle{remark}

\usepackage{comment}

\title{Optimal H\"{o}lder regularity for the $\bar\partial$ problem on product domains in $\mathbb C^2$}

\author{Yuan Zhang}
\date{}

\begin{document}

\maketitle

\begin{abstract}
The  note concerns the $\bar\partial$ problem  on product domains in $\mathbb C^2$. We show that there exists a bounded solution operator  from  $C^{k, \alpha}$ into itself,  $k\in \mathbb Z^+\cup \{0\}, 0<\alpha< 1$. The  regularity result is optimal in view of an example of Stein-Kerzman.
\end{abstract}

\renewcommand{\thefootnote}{\fnsymbol{footnote}}
\footnotetext{\hspace*{-7mm}
\begin{tabular}{@{}r@{}p{16.5cm}@{}}
& 2010 Mathematics Subject Classification. Primary 32W05; Secondary 32A26, 32A55.\\
& Keywords. $\bar{\partial}$ problem, product domains, H\"older spaces, optimal.
\end{tabular}}

\section{ Introduction}

Let $\Omega\subset \mathbb C^n$ be the product of  planar domains whose boundaries consist of finite number of non-intersecting rectifiable Jordan curves. Then $\Omega$ is weakly pseudoconvex  with at most Lipschitz boundary.  A natural question is to look for a  solution operator to the $\bar\partial$ problem on $\Omega$ that achieves the optimal regularity.

As indicated by Example \ref{ex2} of Stein-Kerzman \cite{Kerzman}, the  $\bar\partial$ problem on product domains does not gain  regularity in general. This phenomenon is in sharp contrast with some well-understood domains bearing with nice geometry (such as strict pseudoconvexity, convexity and/or finite types), on which solutions with gained regularity always exist. See  \cite{ DFF, GL,  henkin, HR, Kerzman, LR} et al and the references therein.


Initiated by the work of Henkin \cite{henkin2}  on the bidisc, Bertrams \cite{B}, Chen-McNeal \cite{ChM}\cite{ChM2}, Fassina-Pan \cite{FP} and Jin-Yuan \cite{JY} etc investigated uniform $C^k$   and Sobolev norms of solutions on product domains.  In the H\"older category, the celebrated work of Nijenhuis and Woolf \cite{NW} constructed optimal H\"older solutions in some special {\it iterated} H\"older spaces for polydiscs.   Pan and the author \cite{PZ} recently  proved existence of (the standard) H\"older  solutions  with an infinitesimal loss of H\"older regularity by analysing the parameter dependence of the Cauchy singular integrals.

 In this note, we prove that  for product domains in $\mathbb C^2$, the solution operator in \cite{PZ} must attain the same regularity as that of the H\"older data.  Thus the operator achieves the optimal regularity in view of Example \ref{ex2}. The proof relies on a careful inspection of the H\"older regularity along each direction.

\begin{thm}\label{main}
Let  $\Omega = D_1\times D_2$, where $D_1$ and $D_2$ are two bounded domains in $\mathbb C$ with  $C^{k+1,\alpha}$ boundaries, $ k\in \mathbb Z^+\cup \{0\}, 0<\alpha< 1$. Assume $\mathbf f\in C^{k, \alpha}(\Omega)$ is a $\bar\partial$-closed $(p, q)$ form on $\Omega$  (in the sense of distributions if $k=0$), $0\le p\le 2, 1\le q\le 2$. Then there exists a solution operator $T$ to  $\bar\partial u =\mathbf f$ on $\Omega$ such that  $T\mathbf f\in C^ {k, \alpha}(\Omega)$,  $\bar\partial T\mathbf f = \mathbf f $, and $\|T\mathbf f\|_{C^{k, \alpha}(\Omega)}\le C\|\mathbf f\|_{C^{k, \alpha}(\Omega)}$, where the constant $C$ depends only on $\Omega, k$ and $ \alpha$.
\end{thm}

It is not  clear  whether the same result extends to  general product domains in $\mathbb C^n, n\ge 3$, as Example \ref{ex1} demonstrates.
As a direct consequence of Theorem \ref{main},  the following regularity corollary holds for smooth $(0,1)$ forms up to the boundary.

\begin{cor}\label{mains}
  Let  $\Omega: = D_1\times D_2$, where $D_1$ and $D_2$ are two bounded domains in $\mathbb C$ with  smooth boundaries.  Assume $\mathbf f\in C^{\infty}(\overline\Omega)$ is a $\bar\partial$-closed $(p,q)$ form on $\Omega$,  $0\le p\le 2, 1\le q\le 2$. Then there exists a  solution $u\in C^{\infty}(\overline\Omega)$ to $\bar\partial u =\mathbf f$ on $\Omega$ such that  for each $k\in \mathbb Z^+\cup \{0\}$, $0<\alpha < 1$, $\|u\|_{C^{k, \alpha}(\Omega)}\le C \|\mathbf f\|_{C^{k, \alpha}(\Omega)}$, where the constant $C$ depends only on $\Omega, k$ and $\alpha$.
\end{cor}
\medskip

\noindent\textbf{Acknowledgement:} The author thanks Professor Yifei Pan for valuable suggestions. The author dedicates the paper  to the memory of her father, Baoguo Zhang, who had consistently supported her  in life and work.

\section{Notations and preliminaries}
Let $\Omega$ be an open subset of $\mathbb C^n$. For $0<\alpha<1$, define the ($\alpha$-)H\"older semi-norm of a function $f$ on $\Omega$ to be
$$  H^\alpha[f]: =\sup_{z, z'\in \Omega, z\ne z'} \frac{|f(z) - f(z')|}{|z-z'|^\alpha}. $$ Given any $f\in C^k(\Omega), k\in \mathbb Z^+\cup\{0\}$,  its $C^k$ norm is denoted by $ \|f\|_{C^k(\Omega)}: = \sum_{|\beta|=0}^k\sup_{z\in \Omega}|D^\beta f(z)|$, where $D^\beta$ represents any $|\beta|$-th derivative operator.
 A function $f\in C^k(\Omega)$ is said to be in $  C^{k, \alpha}(\Omega)$ if
     $$\|f\|_{C^{k, \alpha}(\Omega)}: = \|f\|_{C^k(\Omega)} + \sum_{|\beta|=k}H^\alpha[D^\beta f] <\infty.$$
  We say a $(p,q)$ form  to be in $C^{k, \alpha}(\Omega)$ if all its coefficients are in $ C^{k, \alpha}(\Omega)$. When $k=0$, we suppress  $k$ in the notations  by writing $ C^{0, \alpha} (\Omega)$ as $ C^{\alpha} (\Omega)$, and  $ C^{0} (\Omega)$ as $ C (\Omega)$.
\medskip



Assume that $\Omega: = D_1\times\ldots \times D_n$ is a product of planar domains $D_j, 1\le j\le n$. Fixing $(z_1, \ldots, z_{j-1}, z_{j+1}, \ldots, z_n)\in D_1\times \ldots\times D_{j-1}\times D_{j+1}\times \ldots\times D_n$, denote  the H\"older semi-norm of a function $f$ on $\Omega$ along the $z_j$ variable  by
 \begin{equation*}\begin{split}
    &H_j^\alpha[f](z_1, \ldots, z_{j-1}, z_{j+1}, \ldots, z_n):\\
   =&\sup_{\zeta, \zeta' \in D_j, \zeta\ne \zeta'} \frac{|f(z_1, \ldots, z_{j-1},\zeta', z_{j+1}, \ldots, z_n) - f(z_1, \ldots, z_{j-1},\zeta, z_{j+1}, \ldots, z_n)|}{|\zeta'-\zeta|^\alpha}.
   \end{split}
\end{equation*}
 Then one has by triangle inequality that  \begin{equation}\label{eq}
 H^\alpha[f]   \le \sum_{j=1}^n\sup_{\substack{z_l\in D_l\\
    1\le l(\ne j)\le n}}H_j^\alpha[f](z_1, \ldots, z_{j-1}, z_{j+1}, \ldots, z_n).
\end{equation}

Suppose in addition that each slice  $D_j$  of $\Omega$ is bounded with  $C^{k+1,\alpha}$ boundary, $1\le j\le n$. We define the solid and boundary Cauchy integral of a function $f\in C^{k, \alpha}(\Omega)$ along the $z_j$ variable to be
\begin{equation*}
  \begin{split}
    T_j f(z):&=-\frac{1}{2\pi i}\int_{D_j}\frac{f(z_1, \ldots, z_{j-1},\zeta_j, z_{j+1}, \ldots, z_n)}{\zeta_j-z_j}d\bar\zeta_j\wedge d\zeta_j,\ \ \   z\in \Omega; \\
     S_j f(z):&=\frac{1}{2\pi i}\int_{bD_j}\frac{f(z_1, \ldots, z_{j-1},\zeta_j, z_{j+1}, \ldots, z_n)}{\zeta_j-z_j}d\zeta_j, \ \ z\in \Omega.
  \end{split}
\end{equation*}
The classical one-dimensional singular integral theory (see \cite{V}, or \cite[Lemma 4.1]{PZ}) states that for each $1\le j\le n$,
\begin{equation}\label{T1}
    \begin{split}
    \sup_{\substack{z_l\in D_l\\
    1\le l(\ne j)\le n}} H_j^\alpha [D_j^ k T_j f](z_1, \ldots, z_{j-1}, z_{j+1}, \ldots, z_n  ) \lesssim&\left\{
      \begin{array}{cc}
      \|f\|_{C(\Omega)}, & k=0 \\
     \|f\|_{C^{ k-1, \alpha}(\Omega)}, & k\ge 1
     \end{array}
\right. ; \end{split}
  \end{equation}
  \begin{equation}\label{S1}
  \begin{split}
    \sup_{\substack{z_l\in D_l\\
    1\le l(\ne j)\le n}}H_j^\alpha[D_j^ k S_j f](z_1, \ldots, z_{j-1}, z_{j+1}, \ldots, z_n) \lesssim& \|f\|_{C^{ k, \alpha}(\Omega)}.
    \end{split}
  \end{equation}
 Here $D_j^k$ represents $k$-th order derivative operator with respect to the $z_j$ variable, and the two quantities $a$ and $b$ are said to satisfy $a\lesssim b$ if there exists a constant $C$ dependent only on $\Omega, k$ and $\alpha$, such that $a\le Cb$.
 \medskip

 It was further proved in \cite[Theorem 1.1]{PZ} that
 for each $1\le j\le n$, the operator  $T_j$ sends  $C^{k, \alpha}(\Omega)$ into $C^{k, \alpha}(\Omega)$ with
\begin{equation}\label{Tb}
          \|T_j f\|_{C^{k, \alpha}(\Omega)} \lesssim \|f\|_{C^{k, \alpha}(\Omega)}
          \end{equation}
      for any $f\in C^{k, \alpha}(\Omega)$;     and for any small $\epsilon$ with $0<\epsilon<\alpha$, the operator  $S_j$ sends  $C^{k, \alpha}(\Omega)$ into $C^{k, \alpha-\epsilon}(\Omega)$ with
          \begin{equation}\label{Sb}
          \|S_j f\|_{C^{k, \alpha-\epsilon}(\Omega)} \lesssim \|f\|_{C^{k, \alpha}(\Omega)}
          \end{equation}
for any $f\in C^{k, \alpha}(\Omega)$. It is worth mentioning  that  both (\ref{Tb}) and (\ref{Sb}) are sharp estimates (see Example 4.2-4.3 in \cite{PZ}), in the sense that the H\"older regularity in neither inequality can  be further  improved.
\medskip

Finally, given any $\bar\partial$ closed (0,1) form $\mathbf f =\sum_{j=1}^n f_jd\bar z_j\in C^{k, \alpha}(\Omega)$, define  as in \cite{NW} \begin{equation}\label{key}
  T\mathbf f: = T_1f_1+T_2S_1f_2+\cdots+T_nS_1\ldots S_{n-1}f_n.
\end{equation}
 It is not hard to verify that $T$ is a  solution operator to $\bar\partial$ on $\Omega$ (in the sense of distributions if $k=0$),  using the identities $\bar\partial_jT_j = S_j + T_j\bar\partial_j  = id $. As a consequence of (\ref{Tb}) and (\ref{Sb}), the solution operator $T$ achieves the H\"older regularity with at most an infinitesimal loss from that of the data.

\section{ The optimal H\"older estimates  }

Let $\Omega = D_1\times D_2$, where $D_j\subset\mathbb C$ is a bounded domain with  $C^{k+1,\alpha}$ boundary, $j= 1, 2$, $k\in \mathbb Z^+\cup \{0\}, 0<\alpha< 1$.  Despite a loss of H\"older regularity of $S_j$ in $C^{k, \alpha}(\Omega)$ as in (\ref{Sb}),
 the following proposition shows that the composition operator $S_jT_l, j\ne l$ preserves exactly the same H\"older regularity.
 The key observation of the proof is that  the loss of H\"older regularity of $S_j$ only occurs along the $z_l$ direction, which is compensated by a gain of  H\"older regularity of $T_l$ in this same direction.

\medskip

 \begin{prop}\label{HolderS}
  For each $k\in \mathbb Z^+\cup \{0\}$ and $0<\alpha<1$, $ 1\le j\ne l\le 2$, there exists some constant $C$ dependent only on $\Omega, k$ and $\alpha$, such that for any $f\in C^{k, \alpha}(\Omega)$,
      \begin{equation*}
    \|S_jT_l f\|_{C^{k, \alpha}(\Omega)}\le C \|f\|_{C^{k, \alpha}(\Omega)}.
     \end{equation*}
\end{prop}

\begin{proof}
  For simplicity, assume $j=1$ and $l=2$.  Let $\gamma:  =(\gamma_1, \gamma_2)$ with $|\gamma|\le k$. Since $S_1T_2f$ is holomorphic with respect to the $z_1$ variable, we only need to estimate $\|D_2^{\gamma_2}\partial_1^{\gamma_1} S_1T_2f\|_{C^{\alpha}(\Omega)} $.

  Write $ bD_1 = \cup_{m=1}^N \Gamma_m$, where each  Jordan curve $\Gamma_m$  is connected, positively oriented  with respect to $D_1$, and  of  length $s_m$. Let  $\zeta_1|_{ s\in [\sum_{j=1}^{m-1} s_j, \sum_{j=1}^{m} s_j)}$ be  a $C^{k+1, \alpha}$ parametrization of $\Gamma_m$ with respect to the arclength variable $s$, and $s_0 = \sum_{m=1}^{N} s_m$ is the total length of $bD_1$. In particular, $\zeta_1' = 1/\bar\zeta_1'$  on the interval $(\sum_{j=1}^{m-1}s_{j}, \sum_{j=1}^{m} s_j)$ for  each $1\le m\le N$. For any $(z_1, z_2)\in \Omega$,   integration by parts on  $(\sum_{j=1}^{m-1}s_{j}, \sum_{j=1}^{m} s_j)$ for each $1\le m\le N$ gives
  \begin{equation*}
  \begin{split}
   \partial_1 S_1  T_2f({z_1}, {z_2})
   = &\frac{1}{2\pi i}\int_{bD_1}\partial_{z_1} \left(\frac{1} {\zeta_1(s)-{z_1}}\right)T_2f(\zeta_1(s), {z_2})\zeta_1'(s)ds\\
   =&-\frac{1}{2\pi i} \sum_{m=1}^N\int_{\sum_{j=1}^{m-1}s_{j}}^{\sum_{j=1}^{m} s_j}\partial_{s} \left(\frac{1} {\zeta_1(s)-{z_1}}\right)T_2f(\zeta_1(s), {z_2})ds  \\
    =&\frac{1}{2\pi i} \sum_{m=1}^N\int_{\sum_{j=1}^{m-1}s_{j}}^{\sum_{j=1}^{m} s_j} \frac{\partial_{s} \left(T_2f(\zeta_1(s), {z_2})\right)}{\zeta_1(s)-{z_1}}ds \\
       =&\frac{1}{2\pi i} \sum_{m=1}^N\int_{\sum_{j=1}^{m-1}s_{j}}^{\sum_{j=1}^{m} s_j} \frac{T_2\left(\partial_1 f(\zeta_1(s), {z_2})\zeta_1'(s)+ \bar\partial_{1} f(\zeta_1(s), {z_2} )\bar\zeta_1'(s)\right)}{\zeta_1(s)-{z_1}}ds \\
        =&\frac{1}{2\pi i} \sum_{m=1}^N\int_{\sum_{j=1}^{m-1}s_{j}}^{\sum_{j=1}^{m} s_j} \frac{T_2\left(\partial_1 f(\zeta_1(s), {z_2})+ \bar\partial_{1} f(\zeta_1(s), {z_2} )(\bar\zeta_1'(s))^2\right)}{\zeta_1(s)-{z_1}}\zeta_1'(s)ds \\
        =&: \frac{1}{2\pi i} \int_{bD_1} \frac{T_2\tilde f(\zeta_1, z_2)}{\zeta_1-{z_1}}d\zeta_1 = S_1T_2\tilde f(z_1, z_2),
      \end{split}
  \end{equation*}
 where the function $\tilde f \in C^{k-1, \alpha}(\Omega)$ such that $\tilde f(\zeta_1(s), z_2)=\partial_{1} f(\zeta_1(s), {z_2})+ \bar\partial_{1} f(\zeta_1(s), {z_2} )(\bar\zeta_1'(s))^2$ on $[0, s_0)\times D_2$ and $\|\tilde f\|_{C^{k-1, \alpha}(\Omega)}\lesssim \|f\|_{C^{k, \alpha}(\Omega)}$ (see \cite[Lemma 6.38]{GT} on page 137 for the construction of an extension).  Repeating the above process, the proposition is  reduced to  prove for each $ \gamma \in \mathbb Z^+\cup\{0\}, \gamma \le k, 0<\alpha <1$,
  $$\|D^{\gamma}_2 S_1 T_2 f\|_{C^{\alpha}(\Omega)}\lesssim \|f\|_{C^{\gamma, \alpha}(\Omega)}$$
 for all $f\in C^{\gamma, \alpha}(\Omega)$.

Firstly, choose an $\epsilon$   such that $0<\epsilon<\alpha$. Applying the estimates (\ref{Sb}) and (\ref{Tb}) to $S_1 T_2f$, we get $$\|D^{\gamma}_2 S_1 T_2f\|_{C(\Omega)}  \le \|S_1 T_2f\|_{C^{\gamma, \alpha-\epsilon}(\Omega)} \lesssim \|T_2f\|_{C^{\gamma, \alpha}(\Omega)} \lesssim \|f\|_{C^{\gamma, \alpha}(\Omega)}.$$  

We  next verify that $H^\alpha [D^{\gamma}_2 S_1 T_2f] \lesssim \|f\|_{C^{ \gamma, \alpha}(\Omega)}$. 
 Fixing $z_2\in D_2$, since $D^{\gamma}_2 S_1 T_2f = S_1 D^{\gamma}_2 T_2f$,
 $$H_1^\alpha[ D^{\gamma}_2 S_1 T_2f](z_2) = H_1^\alpha[ S_1 D^{\gamma}_2 T_2f](z_2)\lesssim 
   \| D^{\gamma}_2 T_2f\|_{C^\alpha(\Omega)}.$$  Here the last inequality has used  (\ref{S1}) for the estimate of $S_1$ on $D_1$. Consequently, applying  (\ref{Tb}) to the operator $T_2$ in the last term, we obtain $$H_1^\alpha[ D^{\gamma}_2 S_1 T_2f] (z_2) \lesssim \|T_2 f\|_{C^{\gamma, \alpha}(\Omega)} \lesssim \|f\|_{C^{\gamma, \alpha}(\Omega)}.$$

 We further show  for each $z_1\in D_1$, $H_2^\alpha[ D^{\gamma}_2 S_1 T_2f](z_1) \lesssim \|f\|_{C^{\gamma, \alpha}(\Omega)}$. If  $\gamma \ge 1$, making use of   the identity $D^{\gamma}_2 S_1 T_2f =  D^{\gamma}_2 T_2 S_1 f$  by Fubini's theorem, and  the second case of (\ref{T1}) for  $T_2$  along the $z_2$ direction, one deduces
 $$H^\alpha_2[ D^{\gamma}_2 S_1 T_2f](z_1) = H^{\alpha}_2[ D^{\gamma}_2 T_2 S_1 f](z_1)
  \lesssim
  \|S_1f\|_{C^{\gamma-1, \alpha}(\Omega)}.$$
Together with (\ref{Sb}) for $S_1$ on $\Omega$, we infer
$$H^\alpha_2[ D^{\gamma}_2 S_1 T_2f](z_1) \lesssim \|f\|_{C^{\gamma, \alpha}(\Omega)}.$$
 When $\gamma = 0$,   the first case of (\ref{T1}) for  $T_2$   and (\ref{Sb}) for $S_1$ together give
 $$H^\alpha_2[ D^{\gamma}_2 S_1 T_2f](z_1) = H^{\alpha}_2[  T_2 S_1 f](z_1) \lesssim \|S_1f\|_{C(\Omega)} \lesssim \|f\|_{C^{ \alpha}(\Omega)}.$$
 The proof of the proposition is complete in view of (\ref{eq}).
\end{proof}
\medskip

\begin{proof}[Proof of Theorem \ref{main} and Corollary \ref{mains}] We only need to prove the case when $p=0$. If $q=2$, for any datum $ \mathbf f = fd\bar z_1\wedge d\bar z_2$, it is easy to verify that $T_1 f d\bar z_2$ is a solution to $\bar\partial $ on $\Omega$. The optimal H\"older estimate follows from that of $T_1$ operator demonstrated in (\ref{Tb}). For $q=1$, the  H\"older estimate of the solution given by (\ref{key})  is a consequence of (\ref{Tb}) and Proposition \ref{HolderS}, from which the theorem and the corollary follow.
\end{proof}
\medskip

Motivated by an $L^\infty$ example of Stein and Kerzman \cite{Kerzman}, it was shown in \cite{PZ} that the following $\bar\partial$ problem on the bidisc does not gain regularity in H\"older spaces, according to which the H\"older regularity in Theoerem \ref{main} is optimal.

\begin{example}\cite{Kerzman}\label{ex2}
Let $\triangle^2=\{(z_1, z_2)\in \mathbb C^2: |z_1|<1, |z_2|<1\}$ be the bidisc. For each $k\in \mathbb Z^+\cup \{0\}$ and $ 0<\alpha<1$, consider $\bar\partial u =\mathbf f:= \bar\partial ((z_1-1)^{k+\alpha}\bar z_2)$ on $\triangle^2$, $\frac{1}{2}\pi <\arg (z_1-1)<\frac{3}{2}\pi$. Then $\mathbf f\in C^{k,\alpha}(\triangle^2)$ is $\bar\partial$-closed $(0,1)$. However, there does not exist a solution $u\in C^{k, \alpha'}(\triangle^2)$ to $\bar\partial u =\mathbf f$ for any $\alpha'$ with $1>\alpha'>\alpha$.
 \end{example}

Unfortunately, our method does not  obtain optimal H\"older estimates  for product domains of dimension larger than 2. For instance, the solution operator  of the $\bar\partial$ problem on product domains when $n=3$ is in the form of $ T\mathbf f = T_1f_1 +T_2S_1 f_2 +T_3S_1S_2 f_3 $. Yet not all three operators involved on the right hand side of the formula are bounded in $C^{\alpha}(\Omega)$ space. In fact, in the following we adopt an example of Tumanov \cite{Tu} to show that $T_2S_1$ fails to send $C^{\alpha}(\Omega)$ into itself, due to the unboundedness of its H\"older semi-norm along the $z_3$ variable. As a result of this, Proposition \ref{HolderS} holds only when $n=2$.


\begin{example}\label{ex1}
 For $(e^{i\theta},\lambda)\in b\triangle \times \triangle$, let
$$
  \tilde h(e^{i\theta}, \lambda): =
  \left\{
      \begin{array}{cc}
     |\lambda|^\alpha, & -\pi\le \theta \le -|\lambda|^\frac{1}{2}; \\
     \theta^{2\alpha}, & -|\lambda|^\frac{1}{2}\le \theta\le 0; \\
      \theta^\alpha, & 0\le \theta\le |\lambda|;\\
      |\lambda|^\alpha, &|\lambda|\le \theta\le \pi,
    \end{array}
\right.$$
  and  $h$ be a $C^\alpha$ extension of $\tilde h$ onto $\triangle^2$. Define  $f(z_1,z_2, z_3): =h(z_1, z_3)$ for $(z_1,z_2, z_3)\in \triangle^3$. Then $f
 \in C^\alpha(\triangle^3)$. However, $T_2S_1f\notin C^\alpha(\triangle^3)$.
\end{example}

\begin{proof}
It is clear to see that $\tilde h \in C^\alpha(b\triangle\times \triangle)$. For each $z' =(z_1, z_3)\in \triangle^2$, let $h(z'):  = \inf_{w\in b\triangle \times \triangle}\{\tilde h(w)+ M|z'-w|^\alpha\}$, where $M = \|\tilde h\|_{C^\alpha(b\triangle \times \triangle)}$. Then $h\in C^\alpha(\triangle^2)$ is a $C^\alpha$ extension of $\tilde h$ onto $\triangle^2$ and $f\in C^\alpha(\triangle^3).$

In  \cite[Section 3]{PZ2}, it was verified that $H_3^\alpha[S_1h](z_1)$ is unbounded near $1\in b\triangle$, and so $S_1h\notin C^\alpha(\triangle^2)$. On the other hand, making use of the fact that  $  T_21(z) = \bar z_2$, $ z\in  \triangle^3$ (see \cite[Appendix 6.1b]{NW} for instance),
 we get $T_2S_1f(z) =  T_2 1(z) \cdot S_1h(z_1, z_3) =\bar z_2 S_1h(z_1, z_3)$, which does not belong to $C^{\alpha}(\triangle^3).$
\end{proof}
\medskip


\noindent Yuan Zhang, zhangyu@pfw.edu, Department of Mathematical Sciences, Purdue University Fort Wayne,
Fort Wayne, IN 46805-1499, USA

\end{document}